\documentclass[11pt,reqno, a4paper]{amsart}
\usepackage[english]{babel}
\usepackage[utf8]{inputenc}
\usepackage{tikz}
\usepackage{lipsum}
\usepackage{hyperref}
\hypersetup{final=true}
\usepackage{amssymb}
\usepackage{float}
\usepackage{xcolor}
\usepackage{mathtools}
\newtheorem{Theorem}{Theorem}[section]
\newtheorem{Example}[Theorem]{Example}
\newtheorem{Lemma}[Theorem]{Lemma}

\newtheorem{Proposition}[Theorem]{Proposition}
\newtheorem{Definition}[Theorem]{Definition}
\newtheorem{Corollary}[Theorem]{Corollary}
\newtheorem{Conjecture}[Theorem]{Conjecture}

\newtheorem{Remark}[Theorem]{Remark}
\newtheorem{Proof of Theorem 3.1}[Theorem]{Proof of Theorem 3.1}

\newcommand{\Bg}{{\mathrm{Bg}}}
\newcommand{\Dg}{{\mathrm{deg}}}
\newcommand{\tr}{\triangleleft}

 \usepackage{stackengine,scalerel}
\usepackage[toc,page,header]{appendix}

\title[Zircons and smooth Bruhat intervals in symmetric groups]{Zircons and smooth Bruhat intervals in symmetric groups}
\author{Vincent Umutabazi}

\address{Department of Mathematics, Link\"oping University, SE-581 83 Link\"oping, Sweden}

\email{vincent.umutabazi@liu.se}

\begin{document}

\begin{abstract}
In this paper, we prove that if the dual of a Bruhat interval in a Weyl group is a zircon, then that interval is rationally smooth. 
Investigating when the converse holds, and drawing inspiration from conjectures by
Delanoy, leads us to pose two conjectures. If true, they imply that for
Bruhat intervals in type $A$, duals of smooth intervals, zircons, and being
isomorphic to lower intervals are all equivalent.
As a verification, we have checked our conjectures in types $A_n$, $n\leq 8$.

\end{abstract}

\maketitle
\section{Introduction}\label{Sec1}
In \cite{marietti_1}, Marietti introduced the notion of a zircon as a way to generalize Bruhat orders on Coxeter groups. Independently, du Cloux \cite{MR1742435} defined the equivalent concept of completely compressible posets. It is a direct consequence of the definition that every lower (i.e., one which contains the identity element) Bruhat interval is a zircon. 

Delanoy \cite{MR2397355} conjectured that in simply laced Weyl groups, no other zircons appear. That is, if a Bruhat interval is a zircon, it is isomorphic to a lower interval. 

In this paper, it is proved that if $I=[u,w]$ is a Bruhat interval in a Weyl group $W$ and the dual $I^{\ast}$ is a zircon, then $I$ is rationally smooth (i.e., the Schubert variety $X(w)$ is rationally smooth at a point corresponding to $u$). We conjecture that the converse of this statement is true in type $A$. Moreover, we propose a stronger conjecture which is more suitable for testing; see Conjecture \ref{ConjC4} below for the precise statement. In particular it implies, in type $A$, that for a Bruhat interval $I$, $I$ is a zircon if and only if $I$ is isomorphic to a lower interval (possibly in a different type) if and only if $I^{\ast}$ is smooth. By computer testing, we have verified that Conjecture \ref{ConjC4} holds in types $A_n$, $n\leq 8$.

As an outline, this paper is organized as follows. In Section \ref{Sect}, we recall some properties about posets, Coxeter groups and Schubert varieties. In Section \ref{Secr}, we prove the connection between intervals that are zircons and rationally smooth intervals.
Section \ref{sec:main} contains the main results. There we present our main conjectures, deduce consequences, and report on computational evidence.
\section{Preliminaries}\label{Sect}
In this section we recall some properties of partially ordered sets and Coxeter groups that will be used in the sequel. For more information, consult \cite{MR2133266}, \cite{MR1066460}, and  \cite{MR2868112}.  

\subsection{Partially ordered sets} 
Let $P$ be a partially ordered set (or a \emph{poset}). If $P$ contains an element $\hat{0}$ such that $\hat{0}\leq x$ for all $x\in P$, then $\hat{0}$ is the \emph{minimum}. 
Similarly, if $P$ contains an element $\hat{1}$ such that $x\leq \hat{1}$ for all $x\in P$, then $\hat{1}$ is the \emph{maximum}. If $x,y\in P$ are such that $x<y$ and there is no $z\in P$ with $x<z<y$, we say that $x$ is \emph{covered} by $y$ and write $x\tr y$.
An \emph{order ideal} of $P$ is an induced subposet $J\subseteq P$ which  satisfies the property that, for every $z\in J$, all elements below $z$ are also in $J$ (i.e, $y\leq z\in J$ implies $y\in J$). An order ideal having a maximum is called \emph{principal}.
The induced subposet of $P$ defined by $[a,b]:=\lbrace x\in P: a\leq x \leq b \rbrace$ is called a (\emph{closed}) \emph{interval}. By convention, we shall reserve interval notation for intervals with more than one element. That is, if we write $[a,b]$, it is understood that $a<b$.
For two posets $P$ and $Q$, a map $\alpha:P\rightarrow Q $ is \emph{order-preserving} if for all $x,y\in P$, $x\leq y$ in $P$ implies that $\alpha (x)\leq \alpha (y)$ in $Q$. We say that two posets $P$ and $Q$ are \emph{isomorphic}, and write $P\cong Q$, if there is an order-preserving bijection $P\rightarrow Q$ whose inverse is order-preserving.
 \begin{Definition}
 	Let $P$ be a poset. The \emph{dual} of $P$, denoted by $P^{\ast}$, is the poset having the same underlying set as $P$ but with the order relation given by $a\leq b$ in $P^{\ast}$ if and only if $b\leq a$ in $P$.
 \end{Definition}
By a \emph{matching} $M$ on $P$ we mean an involution  $M:P\rightarrow P$ such that for all $x\in P$, we have $M(x)\tr x$ or $x\tr M(x)$.
\begin{Definition}
A matching $M:P\rightarrow P$ is \emph{special} if for all $x,y\in P$ with $x\tr y$, we have $M(x)=y$ or $M(x)< M(y)$.
\end{Definition}
In general, special matchings were invented by Brenti \cite{brenti_1, MR2095476}. If $P$ is an Eulerian poset, \emph{compression labellings} on $P$ were independently introduced by du Cloux \cite{MR1742435}.  These compression labellings are equivalent to special matchings.
\begin{Definition}\label{Zirc}
A \emph{zircon} is a poset $P$, such that for every non-minimal element $x\in P$, the principal order ideal $\{ v\in P: v\leq x \}$ is finite and has a special matching.
\end{Definition}
Thus, if $P$ is a zircon, so is every order ideal of $P$.

Originally, in \cite{marietti_1}, Marietti defined zircons in a different way compared to Definition \ref{Zirc}. However the two definitions are equivalent as was proven in \cite{hult}.
\subsection{Coxeter groups and Schubert varieties} 
A \emph{Coxeter system} is a pair $(W,S)$, where $W$ is the \emph{Coxeter group} generated by a set $S$ of \emph{simple reflections} such that for all $s\neq s'\in S$, $s^{2}=e$, and $(ss')^{m(s,s')}=(s's)^{m(s,s')}=e$ for some $m(s,s')\in \{2,3,\ldots \}\cup \{ \infty \}$. If $m(s,s')=\infty$, it means that there is no relation between $s$ and $s'$. Here $e$ denotes the identity element of $W$.  
Any element $w\in W$ can be written as a product of simple reflections as $w=s_1s_2\cdots s_i$. Let the \emph{length} $\ell(w)$ of $w$ be the smallest $i$ for which such an expression exists. 

\begin{Definition}
A Coxeter system $(W,S)$ is said to be \emph{simply laced} if $m(s,s')\in \lbrace 2,3 \rbrace$ for all $s\neq s'\in S$, otherwise it is \emph{multiply laced}.
\end{Definition}
Recall that finite irreducible Coxeter groups have been classified. They are of types $A_n$, $n\geq 1$, $B_n$, $n\geq 2$, $D_n$, $n\geq 4$, $E_6$, $E_7$, $E_8$, $F_4$, $H_3, H_4$, and $I_{2}(m)$, $m\geq 3$. For example, a symmetric group $S_n$ generated by simple transpositions $s_i=(i,i+1)$, for all $1\leq i \leq n-1$, is a Coxeter group of type $A_{n-1}$. More on the classification can be found in \cite{MR2133266} and  \cite{MR1066460}.

For $ w\in W$, define  the \emph{left descent set} of $w$ as $$D_{L}(w):=\{ s\in S: \ell(sw)<\ell(w)\}.$$
Let  $T:=\lbrace wsw^{-1}: s\in S, w\in W \rbrace$ be the set of \emph{reflections} in $W$. For $u,w\in W$, we write $u\rightarrow w$ if there exists $t\in T$ such that $w=tu$ and $\ell (u)<\ell(w)$.

\begin{Definition}
The \emph{Bruhat order} of $(W,S)$ is the partial order relation on $W$ defined by $u\leq w$ if there is a sequence $u=w_{0}\rightarrow w_{1}\cdots \rightarrow w_{k}=w$.
\end{Definition}
\begin{Definition}
The \emph{Bruhat graph} of $(W,S)$, denoted by $\emph{Bg}(W)$, is the graph whose vertex set is $W$ and whose edge set is $\lbrace \lbrace u,v\rbrace: u\rightarrow v \rbrace$.
\end{Definition}
Notice that we consider Bruhat graphs to be undirected.

Let $[u,w]:=\{ y\in W:u\leq y \leq w\}$ be a \emph{Bruhat interval} in $W$.
The Bruhat graph of $[u,w]$, denoted by $\Bg[u,w]$, is the subgraph of 
$\text{Bg}(W)$ induced by $[u,w]$. The \emph{degree} of a vertex $v$ in $\Bg[u,w]$, denoted by $\Dg_{u,w}(v)$, is defined as the number of edges that are incident to $v$ in $\Bg[u,w]$.

The following lemma, known as the \emph{Lifting property} (see \cite{MR2133266} and \cite{MR435249}), is a fundamental property of the Bruhat order of $W$. We shall employ it in the proofs of some properties appearing in later sections. 
\begin{Lemma}\label{Lft}
Suppose that $x,y\in W$ where $x<y$ and $s\in D_{L}(y)\setminus D_{L}(x) $. Then $x\leq sy$ and $sx\leq y$.
\end{Lemma}
The idea behind the concept of special matchings is to mimic the Lifting property. More precisely, multiplication by any descent element of $x$ is a special matching of 
$[e,x]$. In particular, the Bruhat order on any Coxeter group is a zircon. 

Let $G$ be a semi-simple, simply connected algebraic group over $\mathbb{C}$ where $\mathcal{T}\subset \mathcal{B}$ is a maximal torus contained in a Borel subgroup of $G$. Let $W=N(\mathcal{T})/\mathcal{T}$ be  the Weyl group where $N(\mathcal{T})$ is the normalizer of $\mathcal{T}$ in $G$. Then, $W$ is a finite Coxeter group.
The \emph{flag variety} $G/\mathcal{B}$ is a disjoint union of Schubert cells whose closures are the Schubert varieties. They are indexed by $W$; we write $X(w)$ for the Schubert variety corresponding to $w\in W$.  More on Schubert varieties can be found in \cite{MR1782635}.

\section{Zircons and rationally smooth intervals}\label{Secr}
  From now on, $W$ is a Weyl group.
In this paper, Theorem \ref{Thc} which is due to Carrell and Peterson, is used as the definition of rational smoothness.
\begin{Theorem}\cite{MR1278700}\label{Thc}
Let $[u,w]$ be some Bruhat interval in $W$. Then, $X(w)$ is rationally smooth at a point corresponding to  $u$ if and only if $\mathrm{deg}_{y,w}(y) =\ell(w)-\ell(y)$ for all $y\in [u,w]$.
\end{Theorem}
In \cite{MR1953262}, Carrell and Kuttler showed that if $W$ is a finite simply laced Coxeter group, $X(w)$ is smooth if and only if it is rationally smooth. 
For a Bruhat interval $[u,w]$ in type $A$, $X(w)$ is smooth at a point corresponding to $u$ if and only if $\mathrm{deg}_{u,w}(u ) =\ell(w)-\ell(u)$ (see \cite{MR788771} and \cite{MR1953262}).
 \begin{Remark}\label{Rem}
 By Dyer's \cite[Proposition 3.3]{MR1104786}, the isomorphism type of $[u,w]$ determines that of the Bruhat graph $\Bg[u,w]$. Hence, if $[u,w]$ is isomorphic to $[u',w']$ and $X(w)$ is rationally smooth at a point corresponding to $u$, then $X(w')$ is rationally smooth at a point corresponding to $u'$.
 \end{Remark}
  \begin{Definition}\label{Def1}
 	Let $[u,w]$ be a Bruhat interval in $W$.
 	Then $[u,w]$ is said to be \emph{smooth} (resp. \emph{rationally smooth}) if its associated Schubert variety $X(w)$ is smooth (resp. rationally smooth) at a point corresponding to $u$.
 \end{Definition}
Equivalently, $[u,w]$ is rationally smooth if $P_{u,w}(q)=1$, where $P$ denotes the Kazhdan-Lusztig polynomial. 

Recall that $ [u,w]^{\ast}$ denotes the dual of a Bruhat interval $ [u,w]$ in $(W,S)$ where $u<w$.
\begin{Proposition}\label{Prop} 
If $[u,w]^{\ast}$ is isomorphic to some lower Bruhat interval $[e,x]$, then $[u,w]$ is rationally smooth. 
\end{Proposition}
\begin{proof}
It is well-known that $\Dg_{e,x}(x)=\ell(x)$. Then, 
if $[u,w]^{\ast}$ is isomorphic to some lower Bruhat interval $[e,x]$, we have that $\ell(x)=\Dg_{e,x}(x)=\mathrm{deg}_{u,w}(u)$.  So, $\mathrm{deg}_{y,w}(y)=\ell(w)-\ell(y)$ for every $y\in  [u,w]$.  
Hence $[u,w]$ is rationally smooth. 
\end{proof}
Theorem \ref{Th11} is essentially due to Brenti, Caselli, and Marietti \cite{MR2222360}. In their theorem, it is stated that if $M$ is a special matching of a \emph{lower} Bruhat interval, then $M$ maps reflections to reflections. However, their proof does not require that the interval is lower. Recently, Gaetz and Gao found an independent proof for arbitrary intervals \cite[Theorem 6.3]{GG}.
\begin{Theorem}\label{Th11}
Let $M$ be a special matching of a Bruhat interval $[u,w]$ in $W$. Let also   $x,y\in [u,w]$ be such that $x^{-1}y \in T$. Then $M(x)^{-1}M(y)\in T$.
\end{Theorem}
We now have the following corollary:
\begin{Corollary}\label{Th1}
If $[u,w]$ is a zircon, then $\Dg_{u,w}(w)=\ell(w)-\ell(u)$.
\end{Corollary}
\begin{proof}
Let $M$ be a special matching on a zircon $[u,w]$ and assume that 
$\ell(w)-\ell(u)\geq 2$.
Then $[u,M(w)]$ is a zircon. By induction on the length, we may assume that $\mathrm{deg}_{u,M(w)}(M(w))=\ell(w)-1-\ell(u)$. Hence, $\Dg_{u,w}(M(w))=\ell(w)-\ell(u)$. By Theorem \ref{Th11}, $M$ gives a graph automorphism of $\Bg[u,w]$. Thus $\Dg_{u,w}(M(w))=\Dg_{u,w}(w)$.
\end{proof}
\begin{Corollary}\label{Cor1}
If $[u,w]^{\ast}$ is a zircon, then $[u,w]$ is rationally smooth.
\end{Corollary}
\begin{proof}
Since $[u,w]^{\ast}$ is a zircon, $[y,w]^{\ast}$ is a zircon for all $y\in [u,w]$.
So by Corollary \ref{Th1}, $\mathrm{deg}_{y,w}(y)=\ell(w)-\ell(y)$ for all $y\in [u,w]$. Hence $[u,w]$ is rationally smooth.
\end{proof}

\section{Conjectures on smooth intervals in symmetric groups} \label{sec:main}
Experiments have led us to suspect that, at least in type $A$, the converse of Corollary \ref{Cor1} holds. This is Conjecture \ref{ConjC3} below. If it holds, and conjectures by Delanoy (Conjecture \ref{Conj1} and \ref{Conj2} below) are also true, then a stronger conjecture, Conjecture \ref{ConjC4}, must be satisfied. In this section, we present these connections and show that this last conjecture implies all the others. We also present evidence in the form of computer calculations. 
\begin{Conjecture}\cite{MR2397355}\label{Conj1}
If $[u,w]$ is a zircon in type $A$, $D$ or $E$, then $[u,w]$ is isomorphic to some lower interval $[e,x]$, where $[u,w]$ and $ [e,x]$ are potentially in different types.
\end{Conjecture}
Recall that $D_{L}(w)$ is the left descent set of $w$.
\begin{Conjecture}\cite{MR2397355}\label{Conj2}
 Let  $[u,w]$ be a zircon in type $A$, $D$ or $E$. Then, there exists $s\in D_{L}(w)$ such that either:
 \begin{enumerate}
       \item $sw \not \geq u $, or  %$sw<w$ and 
       \item $su>u$.     %$sw<w$ and  (both follows from $s\in D_{L}(w)$)
 \end{enumerate}
	 	
\end{Conjecture}
In fact, by Delanoy \cite{MR2397355}, Conjecture \ref{Conj2} implies Conjecture \ref{Conj1}.
Now, consider the following proposition which summarizes what we found in the previous section.  
 \begin{Proposition}\label{Pim}
 Let $W$ be a Weyl group. Then:
 \begin{enumerate}
    \item If $[u,w]^{\ast}$ is isomorphic to some lower interval $[e,x]$, possibly from a different system $(W',S')$, then $[u,w]^{\ast}$ is a zircon, \label{imp1} 
     \item If $[u,w]^{\ast}$ is isomorphic to some lower interval $[e,x]$, possibly from a different system $(W',S')$, then $[u,w]$ is rationally smooth, \label{imp2}
      \item If $[u,w]^{\ast}$ is a zircon, then $[u,w]$ is rationally smooth. \label{imp3}
\end{enumerate}
\end{Proposition}
\begin{proof}
The implication (\ref{imp1}) is clear, while the  implication (\ref{imp2}) follows from Proposition \ref{Prop}, 
and the implication (\ref{imp3}) follows from Corollary \ref{Cor1}.
\end{proof}
Investigating the possible converses of these implications in the case of symmetric groups has led us to believe in the following conjecture. 
\begin{Conjecture}\label{ConjC3}
If $[u,w]$ is smooth in type $A_{n}$, then $[u,w]^{\ast}$ is a zircon.
\end{Conjecture}
The truth of this conjecture would, by Proposition \ref{Pim}, imply that smooth Bruhat intervals in type $A$ are the same as duals of zircons. Delanoy's Conjecture \ref{Conj2} would then imply the next conjecture which, as we shall see, itself implies Conjecture \ref{ConjC3} and is suitable for computer testing.
\begin{Conjecture}\label{ConjC4}
 If $[u,w]$ is smooth in type $A_n$, then there exists $s\in S\setminus D_{L}(u)$ such that either :
\begin{enumerate}
	\item $su\not \leq w$, \text{or}\label{Sta1}
	\item $sw<w$. \label{Sta2} %$su>u$ and
\end{enumerate}
\end{Conjecture}
With the aid of a computer (specifically by using SageMath), we have verified Conjecture \ref{ConjC4} for  $n\leq 8$. 
Note that Conjecture \ref{ConjC4} is false in multiply laced types. As a counterexample, the interval $[s_1,s_1s_2s_1]$ is smooth in type $B_2$ (that is, $m(s_1,s_2)=4$), however there is no $s\in S$ that satisfies any of the conditions from Conjecture \ref{ConjC4}. We have no evidence for or against Conjecture \ref{ConjC4} in types $D$ and $E$.

\begin{Proposition}\label{Cl1}
 Conjecture \ref{ConjC4} implies that every smooth interval $[u,w]$ in type $A_{n}$ has a special matching. 
\end{Proposition}
\begin{proof}
Assume that $[u,w]$ is a smooth Bruhat interval but does not have a special matching. Choose $w$ such that $\ell(w)$ is maximal with this property. Since there exists $s\in S$ satisfying condition (\ref{Sta1}) or (\ref{Sta2}), we have two options: Firstly, if condition  (\ref{Sta2}) holds, left multiplication by $s$ is a special matching (see \cite  [Proposition 4.1]{MR2095476}), which is a contradiction. Secondly, if condition (\ref{Sta1}) holds, then by \cite[Proposition 2.7]{MR1742435} and \cite[Proposition B.4(i)]{MR2397355} $[u,w]\cong [su,sw]$. Hence, by Remark \ref{Rem}, $[su,sw]$ is smooth but does not have a special matching. This is a contradiction since, by the lifting property (Lemma \ref{Lft}), $\ell(sw)> \ell(w)$. 
\end{proof}
\begin{Corollary}\label{C:c}
Conjecture \ref{ConjC4} implies Conjecture \ref{ConjC3}.
\end{Corollary}
\begin{proof}
Let $[u,w]^{\ast}$ be the dual of a smooth interval $[u,w]$ in type $A$. Since $[u,w]$ is smooth, for all $u\leq y<w$ the  Bruhat interval $[y,w]$ is smooth. Assuming Conjecture \ref{ConjC4}, Proposition \ref{Cl1} shows that $[y,w]$ has a special matching, and hence $[y,w]^{\ast}$ has a special matching too. Thus $[u,w]^{\ast}$ is a zircon.
\end{proof} 
Recall that any finite Coxeter group has a unique element of maximal length denoted by  $w_0$. Multiplication by $w_0$ is an anti-automorphism of the Bruhat order. That is, $x\leq y \Leftrightarrow xw_0 \geq yw_0$.
\begin{Corollary}
Conjecture \ref{ConjC4} implies Conjecture \ref{Conj2} in type $A$.
\end{Corollary}
\begin{proof}
Let $[u,w]$ be a zircon, and assume Conjecture \ref{ConjC4} holds. By Corollary \ref{Cor1}, $[ww_0, uw_0]\cong [u,w]^{\ast}$ is smooth. Hence, there exists $s\in S\setminus D_{L}(ww_0)$ such that:
\begin{itemize}
\item[(1)] $sww_0 \not \leq uw_0$, or   
\item[(2)] $suw_0 <uw_0$.
\end{itemize}
Thus, $s\in D_{L}(w)$ and 
\begin{itemize}
\item[(1)] $sw \not \geq u$, or   
\item[(2)] $su > u$.
\end{itemize}
\end{proof}
Recalling that Delanoy \cite{MR2397355} showed that Conjecture \ref{Conj2} implies Conjecture \ref{Conj1}, we notice that the hypotheses of the following theorem are satisfied in type $A$ if Conjecture \ref{ConjC4} holds.
\begin{Theorem}\label{Th2}
Suppose that $W$ is of type $A$, $D$, or $E$. If Conjecture \ref{Conj1} and Conjecture \ref{ConjC3} are true in $W$, then the following are equivalent for $u,w\in W$:
	\begin{enumerate}
		\item $[u,w]^{\ast}$ is a zircon,
		\item $[u,w]^{\ast}$ is isomorphic to some lower Bruhat interval,
		\item $[u,w]$ is smooth.
	\end{enumerate}
\end{Theorem}
\begin{proof}

  The implication $(1)\Rightarrow (2)$ follows if Conjecture \ref{Conj1} is true. The converse $(2)\Rightarrow (1)$  holds since every lower Bruhat interval is a zircon.
Corollary \ref{Cor1} shows that $(1)\Rightarrow (3)$.
Finally, Conjecture \ref{ConjC3} asserts $(3)\Rightarrow (1)$.
\end{proof}
Since we have verified Conjecture \ref{ConjC4} in type $A_n$, $1\leq n \leq 8$, the conclusion of the previous theorem holds in these groups:

\begin{Corollary}\label{CTh2}
	Let $[u,w]$ be an interval in $A_n$ where $1\leq n \leq 8$. The following conditions are equivalent:
	\begin{enumerate}
		\item $[u,w]^{\ast}$ is a zircon,
		\item $[u,w]^{\ast}$ is isomorphic to some lower Bruhat interval,
		\item $[u,w]$ is smooth.
	\end{enumerate}
\end{Corollary}
The following example gives an illustration that the lower interval mentioned in item $2$ of Theorem \ref{Th2} and Corollary \ref{CTh2} cannot in general be expected to be found in the same type as $u$ and $w$.
\begin{Example}
Consider the interval $I=[s_1,s_1s_3s_2s_1s_4s_3]$ in type $A_4$, where $m(s_i,s_{i+1})=3$ for $i=1,2,3$, and $m(s_i,s_j)=2$ for $\vert i-j \vert\geq 2$.
Also, consider the interval $I'=[e,s'_3s'_1s'_2s'_4s'_3]$  in type $D_4$, where $m(s'_1,s'_3)=m(s'_2,s'_3)=m(s'_4,s'_3)=3$, and $m(s'_i,s'_j)=2$ otherwise. The interval $I$ is not isomorphic to any lower Bruhat interval in type $A$. 
The reason is that if $\varphi:I\rightarrow L$ were a poset isomorphism and $L$ some lower interval in type $A$, $\varphi(s_1s_3)$ would be a simple reflection below the length three reflections $\varphi(s_1s_3s_4s_3)$, 
$\varphi(s_1s_3s_2s_3)$, and $\varphi(s_3s_2s_1s_3)$. However, in type $A$, every simple reflection is below at most two length three reflections.
But $I^{\ast}$ is smooth, and $I\cong I'$.
\end{Example}

    \bibliographystyle{amsplain}
	\bibliography{proj_4r}
\end{document}